\documentclass[12pt, reqno]{amsart}
\usepackage{amsmath,amssymb,amsthm,graphicx}
\usepackage[T2A]{fontenc}
\usepackage[left=2.5cm,right=2.2cm,top=2.5cm,bottom=3.45cm]{geometry}
\usepackage[matrix,arrow,curve]{xy}

\usepackage[breaklinks]{hyperref}
\hypersetup{
colorlinks = true, 
linkcolor = black, 
citecolor = blue 
}

\let\le\leqslant
\let\ge\geqslant

\let\ot\otimes
\let\mc\mathcal
\DeclareMathOperator{\id}{id}

\DeclareMathOperator{\im}{Im}
\DeclareMathOperator{\GL}{GL}
\DeclareMathOperator{\SL}{SL}

\DeclareMathOperator{\pr}{pr}
\DeclareMathOperator{\inn}{in}
\DeclareMathOperator{\FF}{\textbf{F}}
\DeclareMathOperator{\GG}{\textbf{G}}
\DeclareMathOperator{\PD}{\textbf{PD}}
\DeclareMathOperator{\Rmod}{\textbf{R-mod}}

\DeclareMathOperator{\Mat}{Mat}
\DeclareMathOperator{\diag}{diag}
\DeclareMathOperator{\CC}{\mathcal{C}}
\DeclareMathOperator{\DD}{\mathcal{D}}
\theoremstyle{definition}
\newtheorem{lemma}{Lemma}[section]
\newtheorem{theorem}[lemma]{Theorem}
\newtheorem{remark}[lemma]{Remark}
\newtheorem{corollary}[lemma]{Corollary}

\newtheorem{example}[lemma]{Example}
\newtheorem{proposition}[lemma]{Proposition}
\newtheorem{definition}[lemma]{Definition}

\begin{document}

\author{Alexandr Grebennikov}
\title{Non-surjective Milnor patching diagrams}

\maketitle

\begin{abstract}
\setlength\parindent{0pt}
Milnor patching diagram is essentially the commutative square of rings, over which gluing of finitely generated projective modules is possible in the strongest sense. Necessary and sufficient conditions for a square to be Milnor patching diagram were studied by Milnor, Beauville-Laszlo and Landsburg. We relate this question to determinant-induced factorization in matrix rings to construct a series of non-surjective Milnor patching diagrams, settling the question of Landsburg, and make a step towards the classification of such examples. Also we consider a possible generalization of the notion of Milnor patching diagram to arbitrary subcategories of modules and obtain a classification result for this setting.
\end{abstract}

\section{Introduction}
The central object of study in this work is a commutative square $\mc{F}$ in the category of rings (commutative, unital):
\begin{equation}
\label{fig:R1R2S} 
    \mc{F} = \vcenter{\xymatrix{R \ar[r]^{g_1} \ar[d]^{g_2} & R_1 \ar[d]^{h_1} \\ R_2 \ar[r]^{h_2} & S}}.
\end{equation}
A wide range of properties of such squares were studied by Nashier and Nickols in \cite{NN}. 

\begin{definition}
\label{def:car_sur}
    We call a square \textit{cartesian} if the following sequence of $R$-modules is exact:
    $$
        \xymatrix{0 \ar[r] & R \ar[r]^{\varphi \;\;\;\;\;} & R_1 \oplus R_2 \ar[r]^{\;\;\;\;\; \psi} & S},
    $$
    $$
        \text{ where } \quad \varphi: r \mapsto (g_1(r), g_2(r)), \quad \psi: (r_1, r_2) \mapsto h_1(r_1) - h_2(r_2).
    $$
    We call a square \textit{surjective} if the map $\psi$ is surjective.
\end{definition}

\begin{definition}
    For a commutative square $\mc{F}$ we define the corresponding \textit{category of patching data} $\PD(\mc{F})$:
    \begin{itemize}
        \item Objects: triples $(M_1, M_2, \zeta)$, $M_1$ --- $R_1$-module, $M_2$ --- $R_2$-module, $\zeta: M_1 \ot_{R_1} S \to M_2 \ot_{R_2} S$ --- isomorphism of $S$-modules;
        \item Morphisms between $(M_1, M_2, \zeta)$ and $(M'_1, M'_2, \zeta')$: pairs of $R_i$-linear maps $(g_1, g_2)$, $g_i: M_i \to M'_i$, such that the following diagram commutes:
        $$
            \vcenter{\xymatrix{M_1 \ot_{R_1} S \ar[d]^{\zeta} \ar[r]^{g_1 \ot 1} & M'_1 \ot_{R_1} S \ar[d]^{\zeta'} \\ M_2 \ot_{R_2} S \ar[r]^{g_2 \ot 1} & M'_2 \ot_{R_2} S}}.
        $$
    \end{itemize}

    We say that an object $(M_1, M_2, \zeta) \in \PD(\mc{F})$ has some property $P$ if both $M_1$ and $M_2$ have this property.
    
    There is an obvious functor $\FF_{\mc{F}}: \Rmod \to \PD(\mc{F})$ which acts on objects as
    $$
        \FF_{\mc{F}}: M \mapsto (M \ot_R R_1, M \ot_R R_2, \id_{M \ot_R S}).
    $$
    
    Also define the <<patch>> functor $\GG_{\mc{F}}: \PD(\mc{F}) \to \Rmod$ as
    $$
        \GG_{\mc{F}}: (M_1, M_2, \zeta) \mapsto \{(m_1, m_2) \in M_1 \oplus M_2 : \zeta(m_1 \ot 1) = m_2 \ot 1 \}.
    $$    
\end{definition}

The following classical theorem \ref{Beauville-Laszlo} shows that under certain assumptions the functors $\FF_{\mc{F}}$ and $\GG_{\mc{F}}$ are, in some sense, mutually inverse. 
\footnote{Technically, the statement of \ref{Beauville-Laszlo} is not precisely the theorem from \cite{BL}, but some modification of it, taken from \cite{Stacks}. See section \ref{sec:appendix} (Appendix) of this work for details.}
Despite it having a nice formulation in terms of algebraic geometry as well as numerous applications in that area, for our purposes it is more convenient to stick to commutative-algebraic point of view.


\begin{theorem}[Beauville, Laszlo, \cite{BL}]
\label{Beauville-Laszlo}
    Consider a commutative square $\mc{F}_{loc}$ of the form
    \begin{equation}        
    \label{fig:localization}
        \mc{F}_{loc} = \vcenter{\xymatrix{R \ar[r] \ar[d]^{g} & R_x \ar[d]^{g_x} \\ R' \ar[r] & R'_{g(x)}}},
    \end{equation} 
    where $x \in R$ and $g(x) \in R'$ are not zero divisors (we refer to such squares as \textit{principal localization squares}). Suppose $\mc{F}_{loc}$ is cartesian and surjective. Then $\FF_{\mc{F}_{loc}}$ and $\GG_{\mc{F}_{loc}}$ induce mutually inverse equivalences between the category of $R$-modules with property $P$ and the category of patching data $\PD({\mc{F}_{loc}})$ with property $P$, if $P$ means   
    \begin{enumerate}
        \item[(a)] finitely generated projective;
        \item[(b)] with no $x$-torsion;
        \item[(c)] finitely generated with no $x$-torsion.
    \end{enumerate}
\end{theorem}

\textit{The notations for the rings and the maps in the square we used in diagrams (\ref{fig:R1R2S}) and (\ref{fig:localization}), will be also used throughout the whole paper.}

We will be interested in the inverse statement. Commutative squares satisfying the condition from the point (a) --- namely, such squares $\mc{F}$ that the functors $\FF_{\mc{F}}$ and $\GG_{\mc{F}}$ induce mutually inverse equivalences between the category of finitely generated projective $R$-modules and the category of finitely generated projective patching data --- have already received some attention in the literature under the name \textit{Milnor patching diagrams} since they were first studied by Milnor in \cite{M}. But sometimes it is more convenient to use a slightly stronger and more explicit notion introduced by Landsburg in \cite{L92} (more details in paragraph \ref{par:strong} below).
\begin{definition}
    A square $\mc{F}$ is called a \textit{strong Milnor patching diagram} if it is cartesian and for every $A_S \in \GL_n(S)$ there are $A_1 \in \GL_{2n}(R_1), A_2 \in \GL_{2n}(R_2)$ such that    
    $$
        \begin{pmatrix} A_S & 0 \\ 0 & A^{-1}_S \end{pmatrix} = h_1(A_1) h_2(A_2).
    $$
\end{definition}

Landsburg in \cite{L} established the converse to the point (a) of theorem \ref{Beauville-Laszlo} under the assumption that $R$ is Noetherian and none of its maximal localizations is a regular local ring of dimension $2$, and posed the question about whether this assumption is essential. Theorem \ref{theorem:counterex} we prove in section \ref{sec:counterexample} shows that it actually is, making use of the following definition from \cite{EM}.

\begin{definition}
\label{definition:DIF}
    A ring $R$ is said to have \textit{determinant-induced factorization} if for any $C \in \Mat_n(R)$ and a non-zero-divisor $b \in R$ dividing $\det(C)$ there are $A, B \in \Mat_n(R)$ satisfying
    $$
        \det(B) = b, \; AB = C.
    $$
\end{definition}

\begin{theorem}
\label{theorem:counterex} \;
\begin{enumerate}
    \item [1.] Suppose $R$ has determinant-induced factorization, $x, y \in R$ --- non-zero-divisors, and $y$ is a non-zero-divisor on $R/(x)$. Then the principal localization square
    $$
        \mc{L}(R, x, y) = \vcenter{\xymatrix{R \ar[r] \ar[d] & R_x \ar[d] \\ R_y \ar[r] & R_{xy}}}
    $$
    is a strong Milnor patching diagram.
    \item [2.] If additionally $x$ is a non-zero-divisor on $R/(y)$ and $(x, y) \neq R$, then $\mc{L}(R, x, y)$ is a non-surjective strong Milnor patching diagram. 
\end{enumerate}
\end{theorem}
The condition on $R, x$ and $y$ in \ref{theorem:counterex}.2 is not vacuous since we may take, for instance, $R = k[x, y]$ (more details in paragraph \ref{par:DIF} and example \ref{ex:k[]} below). 

In section \ref{sec:classification} we continue to research on how non-surjective (strong) Milnor patching diagrams may look like and make a step towards their classification in theorem \ref{theorem:strong_patching_equivalence}. 
\begin{theorem}
\label{theorem:strong_patching_equivalence}
    Consider a principal localization square
    $$
        \mc{F}_{loc} = \vcenter{\xymatrix{R \ar[r] \ar[d]^{g} & R_x \ar[d]^{g_x} \\ R' \ar[r] & R'_{g(x)}}}, 
    $$
    and suppose also that $R/(x)$ is a (nonzero) Euclidean domain. Then the following are equivalent
    \begin{enumerate}
        \item the induced map $\bar g: R/(x) \to R'/(g(x))$ is injective and identifies $R'/(g(x))$ with some subring of the fraction field $k(R/(x))$, containing $R/(x)$;
        \footnote{Equivalently, one can say that $R'/(g(x))$ is a domain, $\bar g$ is injective and induces an isomorphism $k(\bar g): k(R/(x)) \to k(R'/(g(x)))$ between the fraction fields.}
        \item $\mc{F}_{loc}$ is a strong Milnor patching diagram.
    \end{enumerate}
\end{theorem}

This theorem together with results of Beauville-Laszlo and Landsburg allows us to complete the classification of principal localization Milnor patching diagrams in the case of $(R, M)$ being a Noetherian local ring and $x \in M \setminus M^2$. 

\begin{corollary}
\label{theorem:local}
    Suppose $R$ is a Noetherian local ring with maximal ideal $M$, $x \in M \setminus M^2$. Then a principal localization square $\mc{F}_{loc}$ is a Milnor patching diagram if and only if
    \begin{itemize}
        \item either $\bar g: R/(x) \to R'/(g(x))$ is a isomorphism;
        \item or $R$ is a regular local ring of dimension $2$, $R'/(g(x))$ is isomorphic to the fraction field $k(R/(x))$, and $\bar g: R/(x) \to R'/(g(x))$ is the standard inclusion. 
    \end{itemize}
\end{corollary}

Section \ref{sec:other} is devoted to the conditions from points (b) and (c) of theorem \ref{Beauville-Laszlo} for a square $\mc{F}$, and, more generally, to the condition <<$\FF_{\mc{F}}$ and $\GG_{\mc{F}}$ induce mutually inverse equivalences between the categories $\mc{C}$ (full subcategory of $\Rmod$) and $\mc{D}$ (full subcategory of $\PD(\mc{F})$)>>. We prove the following theorem \ref{theorem:general_categories} and obtain converse statements to the points (b) and (c) as its consequence (corollary \ref{cor:f-torsion_fin-gen}).
\begin{theorem}
\label{theorem:general_categories}
    Assume $\mc{F}$ is a commutative square, $\mc{C}$ and $\mc{D}$ --- full subcategories of $\Rmod$ and $\PD(\mc{F})$ respectively,
    $\FF_{\mc{F}}$ and $\GG_{\mc{F}}$ induce mutually inverse equivalences between $\mc{C}$ and $\mc{D}$.
    Suppose $R \in \mc{C}$, $(R^2_1, R^2_2, X) \in \mc{D}$ for any $X \in \GL_2(S)$, and also $I \in \mc{C}$ for any ideal $I \subset R$, which is a quotient of some $M \in \mc{C}$. Then $\mc{F}$ is cartesian and surjective.
\end{theorem}

\noindent\textbf{Acknowledgements.}
The work is based on the author's undergraduate thesis defended at Saint-Petersburg State University in 2022. The author is thankful to his advisor Anastasia Stavrova for bringing attention to this problem and helpful discussions.

\section{Preliminaries}
\label{sec:preliminaries}

\subsection{Rings with determinant-induced factorization}
\label{par:DIF}

The work \cite{EM} of Estes and Matijevic we are following in this paragraph is devoted to the study of \textit{Towber rings} (see \cite{LG}, \cite{T} for more information on the subject and \cite{Sm} for a survey on related questions). Theorem 1 of \cite{EM} provides a few equivalent definitions of them, the version (i) being the most convenient for our purposes.
\begin{definition}
    A reduced Noetherian ring $R$ is called \textit{Towber} if for any $C_1, \ldots, C_k \in \Mat_n(R)$ and a non-zero-divisor $b \in R$, such that $b \mid \det(\sum_{i = 1}^k X_i C_i)$ for any $X_i \in \Mat_n(R)$, there are $A_1, \ldots, A_k, B \in \Mat_n(R)$ satisfying
    $$
        \det(B) = b, \; A_i B = C_i.
    $$
\end{definition}

Clearly, any Towber ring has determinant-induced factorization in the sense of definition \ref{definition:DIF} by taking $k = 1$. Moreover, proposition 5 of \cite{EM} shows that the converse also holds when $R$ is a direct product of Krull domains.

Corollary 2 of \cite[section 3]{EM} says that a factorial Noetherian ring is Towber if and only if its global dimension is at most 2 and all finitely generated projective modules over it are free. So, polynomial ring $k[x, y]$ for a field $k$ and regular local rings of dimension 2 are the first non-trivial examples of Towber (and thus having determinant-induced factorization) rings.

\subsection{Properties of principal localization squares}
\label{par:technical}
Here we record two basic statements about principal localization squares. Their proofs are quite straightforward (and partially follow from Proposition 4.2 of \cite{NN} and Proposition 2.1 of \cite{Bh}), thus postponed until section \ref{sec:appendix} (Appendix).
\begin{proposition}
\label{prop:g-injective}
    For a principal localization square $\mc{F}_{loc}$ the following are equivalent:
    \begin{enumerate}
        \item $\mc{F}_{loc}$ is cartesian;
        \item the induced map $\bar g: R/(x) \to R'/(g(x))$ is injective.
    \end{enumerate}
\end{proposition}

\begin{proposition}
\label{prop:g-iso}
    For a principal localization square $\mc{F}_{loc}$ the following are equivalent:
    \begin{enumerate}
        \item $\mc{F}_{loc}$ is cartesian and surjective;
        \item the induced map $\bar g: R/(x) \to R'/(g(x))$ is an isomorphism;
        \item the induced map $\bar g_k: R/(x^k) \to R'/(g(x)^k)$ is an isomorphism for all $k \ge 1$.
    \end{enumerate}
\end{proposition}

\subsection{Milnor patching diagrams}
The following basic proposition is not hard to prove (and it may be viewed as a special case of lemma \ref{lemma:R_in_C} we prove later).
\begin{proposition}
\label{prop:cartesian}
    Any Milnor patching diagram is a cartesian square.
\end{proposition}

Also we state precisely the result of Landsburg mentioned in the introduction.
\begin{theorem}[Landsburg, Proposition 2.3 of \cite{L}]
\label{Landsburg}
    Suppose a square $\mc{F}$ is a Milnor patching diagram, $R$ is Noetherian, and no maximal localization of $R$ is a regular local ring of dimension $2$. Then $\mc{F}$ is surjective. 
\end{theorem}

\subsection{Strong Milnor patching diagrams} 
\label{par:strong}

To check that a square is a strong Milnor patching diagram  we always use the following (trivial) lemma.

\begin{lemma}
\label{lemma:sl}
    Suppose a square $\mc{F}$ is cartesian and for every $A \in \SL_n(S)$ there are $A_1 \in \GL_n(R_1), A_2 \in \GL_n(R_2)$ such that $A = h_1(A_1) h_2(A_2)$. Then $\mc{F}$ is a strong Milnor patching diagram.
\end{lemma}


The results of Milnor's work \cite[Chapter 2]{M} (see also \cite{L92} or introduction of \cite{L} for the notation more similar to the one we use here) imply that 
\begin{itemize}
    \item strong Milnor patching diagrams are actually Milnor patching diagrams;
    \item if a square $\mc{F}$ is cartesian and one of $h_1: R_1 \to S$ and $h_2: R_2 \to S$ is surjective, then $\mc{F}$ is a strong Milnor patching diagram (the assumption is stronger than just $\mc{F}$ being surjective, but in comparison to Beauville-Laszlo theorem this result does not require it to be principal localization square). 
\end{itemize}
In \cite{L92} Landsburg  conjectures that strong Milnor patching is equivalent to the usual one. The next lemma asserts that in certain cases this conjecture is straightforward to verify.

\begin{lemma}
\label{lemma:strong-usual}
    Suppose $\mc{F}$ is a Milnor patching diagram, and every finitely generated projective module over $R$ is free. Then $\mc{F}$ is a strong Milnor patching diagram.
\end{lemma}
\begin{proof}
    For $A \in \SL_n(S)$ consider $N = \GG_{\mc{F}}(R_1^n, R_2^n, A)$. By Milnor patching it is a finitely generated projective $R$-module, hence free. So, $N \simeq R^m$, and
    $$
        (R_1^m, R_2^m, \id_{S^m}) = \FF_{\mc{F}}(R^m) \simeq \FF_{\mc{F}}(N) \simeq (R_1^{n}, R_2^{n}, A)
    $$ 
    in the category $\PD(\mc{F})$ of patching data, what means $n = m$ and 
    $$
        A = h_1(A_1)h_2(A_2), \text{ for some } A_1 \in \GL_{n}(R_1), \; A_2 \in \GL_{n}(R_2).
    $$
    So, we are done by lemma \ref{lemma:sl}.
\end{proof}

\section{Examples of non-surjective Milnor patching diagrams}
\label{sec:counterexample}

\begin{proof}[\textbf{Proof of theorem \ref{theorem:counterex}.1.}]
It's straightforward to check that the square $\mc{L}(R, x, y)$ is cartesian:
$$
    \frac{a}{x^{e_x}} = \frac{b}{y^{e_y}}, \; e_x, e_y \ge 1, \; x \nmid a, \; y \nmid b 
$$
implies $x \mid bx^{e_x} = ay^{e_y}$, together with $y$ being non-zero-divisor on $R/(x)$ it means $x \mid a$ --- a contradiction. 

Consider $C \in \SL_n(R_{xy})$. For sufficiently large $m$
$$
    C' = (xy)^m A \in \Mat_n(R), \; \det(C') = (xy)^{mn}.
$$
Since $R$ has determinant-induced factorization, there exist $A', B' \in \Mat_n(R)$, such that
$$
    \det(A') = x^{mn}, \; \det(B') = y^{mn}, \; A'B' = C'.
$$
Then 
$$
A = x^{-m} A' \in \Mat_n(R_x), \; B = y^{-m} B' \in \Mat_n(R_y), \; AB = C, 
$$
and we are done by \ref{lemma:sl}.
\end{proof}

\begin{proof}[\textbf{Proof of theorem \ref{theorem:counterex}.2.}]
Given \ref{theorem:counterex}.1, it suffices to check that $\frac{1}{xy} \notin R_x + R_y$. Suppose
$$
    \frac{1}{xy} = \frac{a}{x^{e_x}} + \frac{b}{y^{e_y}} = \frac{ay^{e_y} + bx^{e_x}}{x^{e_x}y^{e_y}}, \; e_x, e_y \ge 1, \; x \nmid a, \; y \nmid b, \text{ then }
$$
$$
    x^{e_x-1}y^{e_y-1} = ax^{e_x} + by^{e_y}.
$$
Since $x$ and $y$ are non-zero-divisors on $R/(y)$ and $R/(x)$ respectively, right hand side is non-divisible by $x$ or $y$. Thus $e_x = e_y = 1$, but $1 = ax^{e_x} + by^{e_y}$ contradicts $(x, y) \neq R$. 

\end{proof}

\begin{example}
\label{ex:k[]}
    In partucular, theorem \ref{theorem:counterex}.2 provides the following simple example of a non-surjective Milnor patching diagram (for any field $k$)    
    $$
        \mc{L}(k[x, y], x, y) = \vcenter{\xymatrix{k[x, y] \ar[r] \ar[d] & k[x^{\pm 1}, y] \ar[d] \\ k[x, y^{\pm 1}] \ar[r] & k[x^{\pm 1}, y^{\pm 1}] }}.
    $$
    Though the proof we presented invokes some general theory developed in \cite{EM} to guarantee determinant-induced factorization in $k[x, y]$, for this particular square it may be done by elementary methods (as we do in section \ref{sec:classification} below). 
\end{example}

\section{Partial classification of Milnor patching diagrams}
\label{sec:classification}
In this section we prove theorem \ref{theorem:strong_patching_equivalence} and apply it to corollary \ref{theorem:local}. We begin with the following analogue of determinant-induced factorization for the current setting, useful for the $(1) \Rightarrow (2)$ part of theorem \ref{theorem:strong_patching_equivalence}.

\begin{lemma}
\label{lemma:DIF-via-Smith}
    Suppose $R/(x)$ is a Euclidean domain, $g : R \to R'$ is such that the induced map $\bar g : R/(x) \to R'/(g(x))$ is injective and identifies $R'/(g(x))$ with some subring of the fraction field $k(R/(x))$. Suppose also
    $$
        X \in \Mat_n(R'),\; \det(X) = g(x)a \;\text{ for some } a \in R'.
    $$
    Then there exist $A \in \Mat_n(R), A' \in \Mat_n(R')$, such that
    $$
        \det(A) = x, \; \det(A') = a, \; X = g(A) A'.
    $$
\end{lemma}

Before passing to the proof we record a version of Smith normal form for a matrix over Euclidean domain.
\begin{proposition}
\label{prop:smith}
    Let $Q$ be a Euclidean domain, $Y \in \Mat_n(Q)$. Then there are $U, V \in E_n(Q)$ --- subgroup of $\SL_n(Q)$ generated by transvections (matrices of the form $I_n + \lambda e_{ij}$ with $\lambda \in Q$, $1 \le i \neq j \le n$) --- such that $UYV$ is diagonal.
\end{proposition}

\begin{proof}[\textbf{Proof of lemma \ref{lemma:DIF-via-Smith}}]
    Let $p: R \to R/(x)$ and $p': R' \to R'/(g(x))$ be the quotient maps. Since $\bar g$ identifies $R'/(g(x))$ with some subring inside the fraction field $k(R/(x))$, taking $T \in R/(x)$ equal to the product of all denominators in $p'(X)$ guarantees us
    $$
        \bar g(T) p'(X) = \bar g(Y) \text{ for some } Y \in \Mat_n(R/(x)).
    $$
    By Smith normal form (\ref{prop:smith}) there exist $U_0, V_0 \in E_n(R/(x))$ such that $U_0 Y V_0$ is diagonal. Then
    $$
        D_0 = \frac{1}{\bar g(T)} \bar g(U_0 Y V_0) = \bar g(U_0) p'(X) \bar g(V_0) \text{ is diagonal as well.}
    $$
    At the same time
    $$
        \det(D_0) = \det(p'(X)) = p'(\det(X)) = 0, 
    $$
    thus $D_0$ has a row (say, it's the first row) consisting of zeros only.

    Take $U, V \in E_n(R)$ such that $p(U) = U_0,\; p(V) = V_0$, then
    $$
        p'(g(U) X g(V)) = \bar g(U_0) p'(X) \bar g(V_0) = D_0, 
    $$
    therefore all the elements in the first row of $g(U) X g(V)$ are divisible by $g(x)$. In other words 
    $$
        g(U) X g(V) = \diag(g(x), 1, \ldots, 1) \cdot X' \text{ for some } X' \in \Mat_n(R'). 
    $$
    Then
    $$
        X = g(U^{-1} \diag(x, 1, \ldots, 1)) \cdot \big( X' g(V^{-1}) \big), 
    $$
    and we are done by taking $A = U^{-1} \diag(x, 1, \ldots, 1)$, $A' = X' g(V^{-1})$.
\end{proof}

\begin{proof}[\textbf{Proof of theorem \ref{theorem:strong_patching_equivalence}. (1) $\Rightarrow$ (2)}.]
Since $\bar g$ is injective, $\mc{F}_{loc}$ is cartesian by proposition \ref{prop:g-injective}.

Consider $X \in \SL_n(R'_{g(x)})$. For sufficiently large $m$
$$
    g(x)^m X \in \Mat_n(R'), \; \det (g(x)^m X) = g(x)^{mn}.
$$
Iterative application of lemma \ref{lemma:DIF-via-Smith} provides a decomposition $g(x)^m X = g(A) A'$ with
$$
    A \in \Mat_n(R), \; \det(A) = x^{mn}, \;\;\; A' \in \Mat_n(R'), \; \det(A') = 1.
$$
Hence
$$
    X = g_x(x^{-m} A) \cdot A', \; x^{-m} A \in \SL_n(R_x), \; A' \in \SL_n(R'), 
$$
and we are done by \ref{lemma:sl}.
\end{proof}

\begin{example}
\label{ex:k[[]]}
    Substituting $R = k[x, y], \; R' = k[y^{\pm1}][[x]]$ (for a field $k$) into theorem \ref{theorem:strong_patching_equivalence}, we obtain a strong Milnor patching diagram
    $$
        \vcenter{\xymatrix{k[x, y] \ar[r] \ar[d] & k[x^{\pm1}, y] \ar[d] \\ k[y^{\pm1}][[x]] \ar[r] & k[y^{\pm1}]((x))}}, 
    $$
    which generalizes (in some sense) both example \ref{ex:k[]} and the classic Beauville-Laszlo square from \cite{BL} with $R' = \widehat R = \varprojlim R/(x^n)$.
\end{example}

Now we proceed to prove the other implication of theorem \ref{theorem:strong_patching_equivalence}. For a commutative square $\mc{F}$ (not necessarily principal localization one, thus we use the notation from diagram (\ref{fig:R1R2S})) and any $s \in S$ define the ideal $J(s)$ of $R$ as
$$
    J(s) = \{r \in R \mid h_1(g_1(r)) \cdot s \in h_1(R_1) + h_2(R_2) \}. 
$$
The following lemma, with the proof similar to the one of \cite[Proposition 2.3]{L}, relates these ideals to Milnor patching property.
\begin{lemma}
\label{lemma:JR=R}
    Suppose $\mc{F}$ is a cartesian square, and for some $s \in S$ there are $n \ge 2$, $A_1 \in \GL_n(R_1), \; A_2 \in \GL_n(R_2)$ such that
    $$
        h_1(A_1) T'(s) = h_2(A_2) \; \text{ where } T'(s) = \begin{pmatrix} 1 & s & 0 & \ldots & 0 \\ 0 & 1 & 0 & \ldots & 0 \\ 0 & 0 & * & \ldots & * \\ \vdots & \vdots & \vdots & & \vdots \\ 0 & 0 & * & \ldots & * \end{pmatrix} \in \GL_n(S).
    $$
    Then $J(s) R_1 = R_1, \; J(s) R_2 = R_2$.
\end{lemma}

\begin{remark}
\label{remark:JR=R}
    For example, the condition of lemma \ref{lemma:JR=R} is satisfied for any $s \in S$ when $\mc{F}$ is a strong Milnor patching diagram (straight from the definition and proposition \ref{prop:g-injective}).
\end{remark}

\begin{proof}[\textbf{Proof of lemma \ref{lemma:JR=R}.}]

    Let $c_1, \ldots, c_n, \; d_1, \ldots, d_n$ be the first two columns of $A_1$, $e_1, \ldots, e_n$ -- the first column of $A_2$. Then 
    $$
        h_1(c_i) = h_2(e_i), \;\; h_1(c_i) \cdot s + h_1(d_i) \in h_2(R_2) \text{ for } 1 \le i \le n.
    $$
    Since the square is cartesian, the first condition implies that $c_i = g_1(c'_i)$ for some $c'_i \in R$. Then the second condition transforms into
    $$
        h_1(g_1(c'_i)) \cdot s \in h_1(R_1) + h_2(R_2).
    $$
    Hence $c'_i \in J(s)$ and
    $$
        J(s)R_1 \supset (g_1(c'_1), \ldots, g_1(c'_n)) R_1 = (c_1, \ldots, c_n) R_1 = R_1, 
    $$
    where the last equality follows from $A_1$ being invertible. Similarly, $J(s) R_2 = R_2$.
\end{proof}

\begin{proof}[\textbf{Proof of theorem \ref{theorem:strong_patching_equivalence}. (2) $\Rightarrow$ (1).}]
    To simplify the notation, we use $Q = R/(x), \; Q' = R'/(g(x))$. We need to prove that     
    \begin{itemize}
        \item $\bar g: Q \to Q'$ is injective;
        \item $Q'$ is a domain;
        \item the induced map $k(Q) \to k(Q')$ between the fraction fields is an isomorphism --- or in other words, for any $y' \in Q'$ there is $y \in Q$, such that $\bar g(y) y' \in \im \bar g$.
    \end{itemize}
    
    Since $\mc{F}_{loc}$ is a strong Milnor patching diagram, it is cartesian by proposition \ref{prop:cartesian}, and by proposition \ref{prop:g-injective} the first condition follows.

    For $r' \in R'$ consider the ideal $J = J(r'/g(x)) \subset R$.
    Then for any $j \in J$
    $$
        g(j) \frac{r'}{g(x)} = r'_1 + \frac{g(r_2)}{g(x)^k} \text{ for some } r'_1 \in R', \; r_2 \in R,
    $$
    $$
        g(j) r' - g(x) r'_1 = \frac{g(r_2)}{g(x)^{k-1}}. 
    $$
    Left hand side lies in $R'$, right hand side lies in $g_x(R_x)$, thus they both come from $R$ since the square is cartesian. Then after passing to the quotient $\overline{g(j) r'} \in \im \bar g \subset R'/(g(x))$.

    So, for the third condition it remains to check that $J_1$ --- projection of $J$ onto $Q = R/(x)$ --- is nonzero. Indeed, by lemma \ref{lemma:JR=R} (and remark \ref{remark:JR=R}) $J R' = R'$, hence
    $$
        J_1 Q' = Q'.
    $$
    Since $\mc{F}_{loc}$ is cartesian and $Q \neq 0$, $Q' \neq 0$ as well.

    For the second condition, suppose $a, b \in Q', \; ab = 0$. Again, consider $J_a, J_b \subset Q$ --- projections of $J(a/g(x))$ and $J(b/g(x))$ onto $Q = R/(x)$. Since $Q$ is a Euclidean domain, these ideals are principal, let $y_a$ and $y_b$ be their generators. As above, 
    $$
        \bar g(J_a) a \subset \im \bar g, \; \bar g(J_b) b \subset \im \bar g,
    $$
    $$
        J_a Q' = J_b Q' = Q'.
    $$
    So, $\bar g(y_a), \bar g(y_b)$ are invertible in $Q'$, but $\bar g(y_a) a, \bar g(y_b) b$ lie in the domain $\im \bar g \simeq Q$. Thus one of $a$ and $b$ has to be zero.       
\end{proof}

\begin{remark}
    The proof of implication \ref{theorem:strong_patching_equivalence}.(2) $\Rightarrow$ (1) only used that $Q = R/(x)$ is a PID, not necessarily Euclidean. Moreover, even this condition may be discarded if one assumes in advance that $Q' = R'/(g(x))$ is a domain.
\end{remark}

\begin{proof}[\textbf{Proof of corollary \ref{theorem:local}}]
    First consider the case when $R$ is a regular local ring of dimension $2$. Then $R/(x)$ is a discrete valuation ring; in particular, it is a Euclidean domain and has only two non-trivial localizations --- itself and its fraction field. Since every projective module over a local ring is free, by lemma \ref{lemma:strong-usual} Milnor patching is equivalent to strong Milnor patching, and theorem \ref{theorem:strong_patching_equivalence} concludes the proof.

    The remaining case follows from results of Beauville-Laszlo and Landsburg. Indeed, if $\bar g$ is an isomorphism then by proposition \ref{prop:g-iso} the square $\mc{F}_{loc}$ is cartesian and surjective, and thus a Milnor patching diagram by Beauville-Laszlo theorem \ref{Beauville-Laszlo}. The converse implication is obtained via Landsburg's theorem \ref{Landsburg} (conditions on $R$ allow us to apply it), proposition \ref{prop:cartesian} and, again, proposition \ref{prop:g-iso}. 
\end{proof}

\section{Patching of other categories of modules}
\label{sec:other}
This section is devoted to the proofs of theorem \ref{theorem:general_categories} and corollary \ref{cor:f-torsion_fin-gen}. Throughout the section we assume that the functors $\FF_{\mc{F}}$ and $\GG_{\mc{F}}$, corresponding to the commutative square
$$
    \mc{F} = \vcenter{\xymatrix{R \ar[r]^{g_1} \ar[d]^{g_2} & R_1 \ar[d]^{h_1} \\ R_2 \ar[r]^{h_2} & S}},
$$
induce mutually inverse equivalences between the categories $\mc{C}$ (full subcategory of $\Rmod$) and $\mc{D}$ (full subcategory of $\PD(\mc{F})$).

\begin{lemma}
\label{lemma:R_in_C}
    If $R \in \mc{C}$ then $\mc{F}$ is cartesian.
\end{lemma}
\begin{proof}
    Consider 
    $$
        R' = \{(r_1, r_2) \in R_1 \times R_2 \mid h_1(r_1) = h_2(r_2)\}, 
    $$
    note that it is a unital subring of $R_1 \times R_2$. By commutativity of the square the map $\varphi: R \to R_1 \times R_2, \; \varphi(r) = (g_1(r), g_2(r))$ induces a homomorphism of unital rings $f:R \to R'$. We need to check that $f$ is actually bijective.

    On the other hand, since $R \in \mc{C}$ there is an isomorphism of $R$-modules
    $$
        R \simeq \GG_{\mc{F}}(\FF_{\mc{F}}(R)) = \ker(\psi: R_1 \oplus R_2 \to S; (r_1, r_2) \mapsto h_1(r_1) - h_2(r_2)) = R', 
    $$
    denote it by $\xi: R \to R'$. 
    
    Let $a = \xi(1)$. If $f(r) = 0$ for some $r \in R$, then $\xi(x) = a f(x) = 0$, what implies $x = 0$ by injectivity of $\xi$. So, $f$ is injective.

    For surjectivity it's sufficient to check $a \in \im f$. By surjectivity of $\xi$ there are $r_0, r_1 \in R$ such that
    $$
        \xi(r_0) = af(r_0) = 1, \; \xi(r_1) = af(r_1) = a^2, \text{ then }
    $$
    $$
        0 = (af(r_1) - a^2) f(r_0) = f(r_1) - a, 
    $$
    thus $a = f(r_1)$.
\end{proof}
So, being cartesian is again an easy part. To check surjectivity, we recall the notation used in section \ref{sec:classification}:  for $s \in S$ define 
$$
    J(s) = \{r \in R \mid h_1(g_1(r)) \cdot s \in h_1(R_1) + h_2(R_2) \} \text{ --- ideal of } R,
$$
$$
    T(s) = \begin{pmatrix} 1 & s \\ 0 & 1 \end{pmatrix} \in \GL_2(S).
$$


\begin{proof}[\textbf{Proof of theorem \ref{theorem:general_categories}.}]
    The square $\mc{F}$ is cartesian by lemma \ref{lemma:R_in_C}, therefore
    $$
        \{ (r_1, r_2) \in R_1 \oplus R_2 \mid h_1(r_1) = h_2(r_2) \} = \{(g_1(r), g_2(r)) \mid r \in R\} \simeq R.
    $$
    
    Surjectivity is equivalent to $J = J(s)$ being equal to $R$ for any $s \in S$. Consider the short exact sequence of patching data given by
    $$
        \xymatrix{0 \ar[r] & R_i \ar[r]^{\inn_1} & R_i \oplus R_i \ar[r]^{\pr_2} & R_i \ar[r] & 0} \;\; (i = 1, 2)
    $$
    with patching along
    $$
        \xymatrix{0 \ar[r] & S \ar[d]^{\id} \ar[r]^{\inn_1} & S^2 \ar[r]^{\pr_2} \ar[d]^{T(s)} & S \ar[r] \ar[d]^{\id} & 0 \\ 
        0 \ar[r] & S \ar[r]^{\inn_1} & S^2 \ar[r]^{\pr_2} & S \ar[r] & 0}.
    $$    
    We can write $\GG_{\mc{F}}(R_1^2, R^2_2, T(s))$ explicitly as
    $$
        N = \GG_{\mc{F}}(R_1^2, R_2^2, T(s)) = 
    $$
    $$
        = \{(x_1, y_1, x_2, y_2) \in R_1^2 \oplus R_2^2: h_1(x_1) + h_1(y_1) s = h_2(x_2),\; h_1(y_1) = h_2(y_2) \}.
    $$
    Then $\GG_{\mc{F}}(R_1, R_2, \id_S) \simeq R$ is a submodule of $N$, given by $y_1 = y_2 = 0$, and quotient of $N$ by this submodule is
    $$
        \{(y_1, y_2) \in R_1 \oplus R_2 : h_1(y_1) = h_2(y_2), \; h_1(y_1) \cdot s \in h_1(R_1) + h_2(R_2)\} = 
    $$
    $$
        = \{(g_1(r), g_2(r)) \mid r \in R, \; h_1(g_1(r)) \cdot s \in h_1(R_1) + h_2(R_2)\} \simeq J \subset R. 
    $$
    In other words, there is a short exact sequence of $R$-modules
    \begin{equation}
    \label{eq:ses}
        \xymatrix{0 \ar[r] & R \ar[r]^{\GG(\inn_1)} & N \ar[r] & J \ar[r] & 0}.
    \end{equation}
    Both $(R_1, R_2, \id_S) = \FF_{\mc{F}}(R)$ and $(R^2_1, R^2_2, T(s))$ lie in $\mc{D}$, hence application of $\FF_{\mc{F}}$ to the sequence (\ref{eq:ses}) results into the following diagrams (for $i = 1, 2$)
    \begin{equation}   
    \label{eq:5-lemma}
        \vcenter{\xymatrix{ & R_i \ar[rr]^{(\FF\circ\GG)(\inn_1)\;\;} \ar[d] & & N \ot_R R_i \ar[d]^{\nu_i} \ar[r] & J \ot_R R_i \ar[r] \ar@{-->}[d]^{\eta_i} & 0 \\
        0 \ar[r] & R_i \ar[rr]^{\inn_1} & & R_i \oplus R_i \ar[r]^{\pr_2} & R_i \ar[r] & 0}},
    \end{equation}
    where the left and the middle vertical arrows are isomorphisms arising from natural transformation between $\FF_{\mc{F}} \circ \GG_{\mc{F}}$ and $\id_{\mc{D}}$. This allows us to notice that the map $R_i \to N \ot_R R_i$ is injective and to define the isomorphisms $\eta_i: J \ot_R R_i \to R_i$ (right vertical arrows).

    Since the middle vertical arrows $\nu_i$ of (\ref{eq:5-lemma}) come from morphisms of patching data, the diagram below is commutative
    $$
        \vcenter{\xymatrix{N \ot_R S \ar[r]^{\id} \ar[d]^{\nu_1 \ot_{R_1} \id_S} & N \ot_R S \ar[d]^{\nu_2 \ot_{R_2} \id_S} \\
        S^2 \ar[r]^{T(s)} & S^2}}, 
    $$
    then the similar diagram with the right vertical arrows of (\ref{eq:5-lemma}) is commutative as well
    $$
        \vcenter{\xymatrix{J \ot_R S \ar[r]^{\id} \ar[d]^{\eta_1 \ot_{R_1} \id_S} & J \ot_R S \ar[d]^{\eta_2 \ot_{R_2} \id_S} \\
        S \ar[r]^{\id} & S}}.
    $$
    So, we constructed an isomorphism between $\FF_{\mc{F}}(J)$ and $(R_1, R_2, \id_S) = \FF_{\mc{F}}(R)$ in the category of patching data $\PD(\mc{F})$. It remains to notice that statement of the theorem guarantees both $J$ (as a quotient of $N$) and $R$ to lie inside $\mc{C}$, thus $J \simeq R$ as $R$-modules.

    In particular, this means that $J$ is generated by one element (call it $y_s$) and that the sequence (\ref{eq:ses}) splits, allowing us to write 
    $$
        N \simeq R \oplus J \simeq R^2.
    $$
    Then application of $\FF_{\mc{F}}$ gives
    $$
        (R^2_1, R^2_2, T(s)) \simeq \FF_{\mc{F}} (N) \simeq \FF_{\mc{F}} (R^2) \simeq (R^2_1, R^2_2, \id_{S^2})
    $$
    in the category $\PD(\mc{F})$ of patching data, what puts us in the setting of lemma \ref{lemma:JR=R} (with $n = 2$). This lemma tells that $g_1(y_s)$ and $g_2(y_s)$ are invertible in $R_1$ and $R_2$, respectively. Since the square $\mc{F}$ is cartesian, $y_s$ is then invertible in $R$, therefore $J = R$.
\end{proof}

Now it is not hard to deduce the converse stements to points (b) and (c) of Beauville-Laszlo theorem \ref{Beauville-Laszlo}.
\begin{corollary}
\label{cor:f-torsion_fin-gen}
    Assume that for a principal localization square 
    $$
        \mc{F}_{loc} = \vcenter{\xymatrix{R \ar[r] \ar[d]^{g} & R_x \ar[d]^{g_x} \\ R' \ar[r] & R'_{g(x)}}}
    $$
    corresponding functors $\FF_{\mc{F}_{loc}}$ and $\GG_{\mc{F}_{loc}}$ induce mutually inverse equivalences between category $\CC$ of $R$-modules with property $P$ and category $\DD$ of patching data with property $P$, if $P$ means
    \begin{enumerate}
        \item [(b)] with no $x$-torsion;
        \item [(c)] finitely generated with no $x$-torsion.
    \end{enumerate}
    Then the square $\mc{F}_{loc}$ is cartesian and surjective.
\end{corollary}
\begin{proof}
    We check the conditions needed to apply theorem \ref{theorem:general_categories}.
    Since $x$ is a non-zero-divisor in $R$, projective modules and ideals have no $x$-torsion. Thus point (b) follows, for point (c) it remains to notice that if $I$ is a quotient of some $M \in \mc{C}$, then it is finitely generated and therefore lies in $\mc{C}$ itself.
\end{proof}

\section{Appendix}
\label{sec:appendix}

The proof of the version \ref{Beauville-Laszlo} of Beauville-Laszlo theorem is essentially written in \cite{Stacks}, though it uses somewhat different notation. To clarify this inconsistency we sketch the proof of theorem \ref{Beauville-Laszlo} with references to statements from \cite{Stacks}. 
\begin{proof}[\textbf{Proof of theorem \ref{Beauville-Laszlo}.}]
    Proposition \ref{prop:g-iso}.(1) $\Rightarrow$ (3). implies that $(R \to R', g)$ is a \textit{glueing pair} in the sense of \cite{Stacks}. Then theorem \cite[tag 0BP2]{Stacks} states that $\FF_{\mc{F}_{loc}}$ and $\GG_{\mc{F}_{loc}}$ provide an equivalence beteween the category of \textit{glueable} $R$-modules and the category of patching data. For any $R$-module $M$ there is an exact sequence
    $$
        \xymatrix{M \ar[r] & (M \ot_R R_x) \oplus (M \ot_R R') \ar[r] & M \ot_R R'_{g(x)} \ar[r] & 0 },
    $$
    and in the case when $M$ has no $x$-torsion point (2) of lemma \cite[tag 0BNW]{Stacks} the left arrow is injective as well. Therefore, all $R$-modules with no $x$-torsion are \textit{glueable}. 
    
    Notice that $M \ot_R R'_{g(x)}$ always has no $x$-torsion, then $M$ having no $x$-torsion together with the short exact sequence above implies that $\FF_{\mc{F}_{loc}}(M)$ has no $x$-torsion as well. If $M_1$ and $M_2$ have no $x$-torsion, then $\GG_{\mc{F}_{loc}}(M_1, M_2, *)$ also does not have it as a submodule of $M_1 \oplus M_2$. Summing it up, we obtain point (b). Points (c) and (a) follow in the similar way from lemmas \cite[tag 0BNN]{Stacks} and \cite[tag 0BP6]{Stacks}, respectively.
\end{proof}

Also we present the technical proofs omitted in paragraph \ref{par:technical}.
\begin{proof}[\textbf{Proof of proposition \ref{prop:g-injective}.}]
     Recall the notation of maps $\varphi$ and $\psi$ from definition \ref{def:car_sur}.
     
     (1) $\Rightarrow$ (2): Let $r \in R, \; g(r) = ag(x)$ for some $a \in R'$. Then $\psi((a, r/x)) = 0$, hence $(a, r/x) \in \im \varphi$. Therefore, $x \mid r$.
     
     (2) $\Rightarrow$ (1): Injectivity of $\varphi$ follows from injectivity of $R \to R_x$, so we verify $\ker \psi = \im \varphi$. Let $(r', r/x^k) \in \ker \psi$ for some $r' \in R', \; r \in R, \; k \ge 1$. Then $g(x)^k r' = g(r)$, hence $g(x) \mid g(r)$, and by injectivity of $\bar g$ \; $x \mid r$, therefore one can decrease $k$. So, we my assume $k = 0$, and then $(r', r) = (g(r), r) = \varphi(r)$.
\end{proof}

\begin{proof} [\textbf{Proof of proposition \ref{prop:g-iso}.}]
    (1) $\Rightarrow$ (2): Injectivity of $\bar g$ follows from the previous proposition, so we verify surjectivity. By surjectivity of the square for each $r' \in R'$ there exist $r'_1 \in R', \; r_2 \in R, \; m \ge 0$, such that
    $$
        \frac{r'}{g(x)} = r'_1 + \frac{r_2}{g(x)^m}.
    $$
    Multiplying by $g(x)$ gives
    $$
        r' - g(x) r'_1 = \frac{r_2}{g(x)^{m-1}}.
    $$
    LHS lies inside $R'$, RHS lies inside $g_x(R_x)$, thus the square being cartesian implies that both of them are inside $g(R)$. Considering the quotient by $g(x)$, we obtain that $\bar r' \in \im \bar g$.
    
    (2) $\Rightarrow$ (3): Induction on $k$, the base case for $k = 1$ is clear. Let $g(r) = g(x)^k r'$, then by induction hypothesis $r = x^{k-1} r_1$. Since $g(x)$ is not a zero divisor, $g(r_1) = g(x) r'$. Then injectivity of $\bar g$ implies that $x \mid r_1$, and hence $x^k \mid r$. So, $\bar g_k: R/(x^k) \to R'/(g(x))^k$ is injective.
    
    Now we verify surjectivity. By induction hypothesis for $r'$ it holds that $r' = g(x)^{k-1} r'_1 + g(r_1)$ $ \; (r'_1 \in R', r_1 \in R)$, and by induction hypothesis for $r'_1$ it holds that $r'_1 = g(x)^{k-1} r'_2 + g(r_2)$ $ \; (r'_2 \in R', r_2 \in R)$. Therefore,
    $$
        r' = g(x)^{2k-2} r'_2 + g(x^{k-1} r_2 + r_1),
    $$
    what gives us surjectivity for each $k \ge 2$.
    
    (3) $\Rightarrow$ (1): The square is cartesian by the previous proposition \ref{prop:g-injective}, so we verify its surjectivity. Indeed, by surjectivity of $\bar g_k$ for any $r' \in R'$ there exist $r'_1 \in R', \; r_2 \in R$, such that $r' = g(x)^k r'_1 + r_2$. Therefore
    $$
        \frac{r'}{g(x)^k} = r'_1 + \frac{r_2}{g(x)^k} \in \im \psi.
    $$
\end{proof}

\bigskip

\noindent{Alexandr Grebennikov\\
IMPA, Rio de Janeiro, Brazil}\\
Saint-Petersburg State University, Saint Petersburg, Russia \\
{\tt sagresash@yandex.ru}

\end{document}